\pdfoutput=1
\RequirePackage{fix-cm}
\documentclass{amsart}
\usepackage{microtype}
\usepackage[utf8]{inputenc}
\usepackage[T1]{fontenc}
\usepackage{mathtools,amssymb}
\usepackage[foot]{amsaddr}
\usepackage{amsthm,thmtools}
\usepackage{enumitem,booktabs}
\mathtoolsset{centercolon,mathic}
\usepackage{listings}
\usepackage{xcolor,xspace}

\usepackage{tikz-cd}
\usetikzlibrary{positioning}

\usepackage{url,cite}

\makeatletter
\protected\def\tikz@nonactivecolon{\ifmmode\mathrel{\mathop\ordinarycolon}\else:\fi}
\makeatother

\definecolor{keywordcolor}{rgb}{0.7, 0.1, 0.1}   
\definecolor{tacticcolor}{rgb}{0.1, 0.2, 0.6}    
\definecolor{commentcolor}{rgb}{0.4, 0.4, 0.4}   
\definecolor{symbolcolor}{rgb}{0.0, 0.1, 0.6}    
\definecolor{sortcolor}{rgb}{0.1, 0.5, 0.1}      

\declaretheorem[numberwithin=section]{theorem}
\declaretheorem[sibling=theorem]{lemma}

\declaretheorem[sibling=theorem]{corollary}
\declaretheorem[sibling=theorem]{conjecture}
\declaretheorem[sibling=theorem,style=definition]{definition}

\newcommand*{\Z}{\mathbb{Z}}
\newcommand*{\susp}[1]{S{#1}}
\newcommand*{\join}[2]{{#1}*{#2}}
\newcommand*{\gluesym}{{\mathsf{glue}}}
\newcommand*{\jglue}{\ensuremath\gluesym\xspace}
\newcommand*{\glue}{\ensuremath\gluesym\xspace}
\newcommand*{\meridsym}{{\mathsf{merid}}}
\newcommand*{\merid}{\ensuremath\meridsym\xspace}
\newcommand*{\define}[1]{\emph{#1}}
\newcommand*{\rev}[1]{{#1}^{-1}}
\newcommand*{\defeq}{:=}
\newcommand*{\Sn}{\mathbb{S}}
\newcommand*{\inl}{{\mathsf{inl}}}
\newcommand*{\inr}{{\mathsf{inr}}}
\newcommand*{\ap}{{\mathsf{ap}}}
\newcommand*{\north}{{\mathsf{N}}}
\newcommand*{\south}{{\mathsf{S}}}
\DeclarePairedDelimiter\norm{\lVert}{\rVert}
\newcommand{\blank}{\mathord{\hspace{1pt}\text{--}\hspace{1pt}}}

\title{The Cayley-Dickson Construction in Homotopy Type Theory}
\date{\today}
\author{Ulrik Buchholtz}
\author{Egbert Rijke}
\address{Department of Philosophy, Carnegie Mellon University,
  Pittsburgh, PA 15213, USA}
\email{\{ulrikb,erijke\}@andrew.cmu.edu}

\begin{document}

\begin{abstract}
We define in the setting of homotopy type theory an H-space structure on
$\Sn^3$. Hence we obtain a description of the quaternionic Hopf fibration
$\Sn^3\hookrightarrow\Sn^7\twoheadrightarrow\Sn^4$, using only homotopy invariant tools.
\end{abstract}

\maketitle

\section{Introduction}

Homotopy type theory is the study of a range of homotopy theoretical
interpretations of Martin-L\"of dependent type theory
\cite{MartinLof1984} as well as an
exploration in doing homotopy theory inside type theory
\cite{TheBook}. This paper concerns the latter aspect, in
that we give a purely type theoretic definition of the quaternionic
Hopf fibration
\begin{equation*}
\begin{tikzcd}
\Sn^3 \arrow[hook]{r} & \Sn^7 \arrow[two heads]{r} & \Sn^4.
\end{tikzcd}
\end{equation*}

Classically, the $3$-sphere can be given the H-space structure given by
multiplication of the quaternions of norm $1$.
For any H-space $A$, the Hopf construction produces a fibration
\begin{equation*}
\begin{tikzcd}
A \arrow[hook]{r} & \join{A}{A} \arrow[two heads]{r} & \susp{A},
\end{tikzcd}
\end{equation*}
where $\join{A}{A}$ denotes the \emph{join} of $A$ with itself, and $\susp{A}$
denotes the \emph{suspension} of $A$.
Hence we get the quaternionic Hopf fibration from the H-space structure on
$\Sn^3$ and the Hopf construction.
The Hopf construction has already been developed in homotopy type
theory \cite[Theorem~8.5.11]{TheBook}, so what is needed is to construct the
H-space structure on $\Sn^3$.

When doing homotopy theory in homotopy type theory, we reason directly with (homotopy) types
and not with any mediating presentation of these, e.g.,
as topological spaces or simplicial sets. For example, the spheres are
defined as iterated suspensions of the empty type (which represents
the $(-1)$-sphere), rather than subsets of the Euclidean spaces
$\mathbb R^n$. Therefore, we cannot directly reproduce the classical construction
of the H-space structure on $\Sn^3$, viewed as a subspace of the quaternions.

Instead, our approach in this paper is to find a type-theoretic incarnation of
the Cayley-Dickson construction that is used to form the classical
algebras $\mathbb R$, $\mathbb C$, $\mathbb H$, $\mathbb O$.

The rest of this paper is organized as follows. In
Section~\ref{sec:cayley-dickson} we recall the
classical Cayley-Dickson construction. We then recall the story of the
Hopf construction in homotopy type theory in Section~\ref{sec:hopf}. In the main
section, Section~\ref{sec:imaginaroids}, we discuss how to port the
Cayley-Dickson construction to type theory in order to construct the
H-space structure on $\Sn^3$. Having performed the construction in homotopy type theory,
in Section~\ref{sec:semantics} we discuss the range of models in which
the construction can be performed. We conclude in
Section~\ref{sec:conclusion}.

Our construction has been formalized and checked in the Lean proof
assistant \cite{Moura2015}.
In fact, we developed the results while formalizing them, and the
proof assistant was helpful as a tool to develop the mathematics.
The full code listing of our formalization is
available in the appendix.

Before we begin, let us address a possible route to our result which
we have not taken: It might seem as if the best way to reason about
quaternions and related algebraic structures in homotopy type theory would be to
construct them in the usual set-theoretic way but such that we could
still access, say, the underlying homotopy type of the unit
sphere. Indeed, this would be possible to do in cohesive homotopy type theory
\cite{Shulman2015}. However, as of now there is no known
interpretation of cohesive homotopy type theory into ordinary homotopy type theory preserving homotopy
types, so this would not give a construction in ordinary homotopy type theory. And even
if such an interpretation were possible, it might require more
machinery to develop than what is used here.

\subsection*{Acknowledgements}

The authors gratefully acknowledge the support of the Air Force Office
of Scientific Research through MURI grant FA9550-15-1-0053. Any
opinions, findings and conclusions or recommendations expressed in
this material are those of the authors and do not necessarily reflect
the views of the AFOSR.

%
%

\section{The classical Cayley-Dickson construction}
\label{sec:cayley-dickson}

Classically, the $1$-, $3$- and $7$-dimensional spheres are subspaces of
$\mathbb{R}^2$, $\mathbb{R}^4$ and $\mathbb{R}^8$, respectively. Each
of these vector spaces can be given the structure of a normed division
algebra, and we get the complex numbers $\mathbb{C}$, Hamilton's
quaternions $\mathbb{H}$, and the octonions $\mathbb{O}$ of Graves and
Cayley. Since, in each of these algebras, the product preserves norm,
the unit sphere is a subgroup of the multiplicative group.

Cayley's construction of the octonions was later generalized by
Dickson \cite{Dickson1919}, who gave a uniform procedure for
generating each of these algebras from the previous one. The process
can be continued indefinitely, giving for instance the $16$-dimensional
sedenion-algebra after the octonions.

Here we describe one variant of the Cayley-Dickson construction,
following the presentation in \cite{Baez2002}. For this purpose, let
an \emph{algebra} be a vector space $A$ over $\mathbb R$ together with
a bilinear multiplication, which need not be associative, and a unit
element $1$. A \emph{$*$-algebra} is an algebra equipped with a linear involution
$*$ (called the \emph{conjugation}) satisfying $1^*=1$ and $(ab)^*=b^*a^*$.

If $A$ is a $*$-algebra, then $A' \defeq A\oplus A$ is again a
$*$-algebra using the definitions
\begin{equation}
  \label{eq:classical-cd}
  (a,b)(c,d) := (ac - db^*, a^*d + cb),
  \quad 1 := (1,0),
  \quad (a,b)^* := (a^*,-b).
\end{equation}
If $A$ is \emph{nicely normed} in the sense that (i) for all $a$, we have $a+a^*\in \mathbb R$
(i.e., the subspace spanned by $1$), and (ii) $aa^*=a^*a>0$ for
nonzero $a$, then so is $A'$. In the nicely normed case, we get a norm by defining
\[\norm a = aa^*,\]
and we have inverses given by $a^{-1}=a^*/\norm a$.
By applying this construction repeatedly, starting with $\mathbb{R}$, we obtain the
following sequence of algebras, each one having slightly fewer good
properties than the preceding one:
\begin{itemize}
\item $\mathbb R$ is a \emph{real} (i.e., $a^*=a$) commutative associative
  nicely normed $*$-algebra,
\item $\mathbb C$ is a commutative associative nicely normed $*$-algebra,
\item $\mathbb H$ is an associative nicely normed $*$-algebra,
\item $\mathbb O$ is an \emph{alternative} (i.e., any subalgebra generated by
  two elements is associative) nicely normed $*$-algebra,
\item the sedenions and the following algebras are nicely normed
  $*$-algebras, which are neither commutative, nor alternative.
\end{itemize}
Being alternative, the first four are normed division algebras, as
$a,b,a^*,b^*$ are in the subalgebra generated by $a-a^*$ and $b-b^*$,
so we get
\[ \norm{ab}^2 = (ab)(ab)^* = (ab)(b^*a^*) = a(bb^*)a^* = \norm
  a^2\norm b^2.\]
However, starting with the sedenions, this fails and we get nontrivial
zero divisors. In fact, the zero divisors of norm one in the sedenions
form a group homeomorphic to the exceptional Lie group $G_2$.

To sum up the story as it relates to us, we first form the four normed
division algebras $\mathbb R$, $\mathbb C$, $\mathbb H$ and $\mathbb
O$ by applying the Cayley-Dickson construction starting with $\mathbb
R$, and then we carve out the unit spheres and get spaces with
multiplication $\Sn^0$, $\Sn^1$, $\Sn^3$ and $\Sn^7$.

In homotopy type theory, we cannot use this strategy directly. Before we
discuss our alternative construction, let us recall some basics
regarding H-spaces in homotopy type theory.

\section{H-spaces and the Hopf construction}
\label{sec:hopf}

First, let us briefly recall that the (homotopy) pushout $A \sqcup^C B$ of a span
\begin{equation*}
\begin{tikzcd}
A & C \arrow{l}[swap]{f} \arrow{r}{g} & B
\end{tikzcd}
\end{equation*}
can be modeled in homotopy type theory as a
higher inductive type with \emph{point constructors}
$\inl : A \to A \sqcup^C B$ and $\inr : B \to A \sqcup^C B$ and a \emph{path
constructor} $\glue : \Pi_{(c : C)}\,\inl(f(c)) = \inr(g(c))$.
The \emph{suspension} $\susp A$ of a type $A$ is the pushout of the span
$1 \leftarrow A \rightarrow 1$, which is equivalently described as the
higher inductive type with point constructors $\north$ and $\south$
(corresponding to left and right injections) and path constructor
$\merid : A \to (\north = \south)$, generating the meridians in the
suspension. The \emph{join} $A*B$ of two types $A$ and $B$ is the pushout of
the span $A \leftarrow A\times B\to B$ given by the projections, which
is equivalently described as the higher inductive type with point
constructors $\inl:A \to A*B$, $\inr: B\to A*B$ and path constructor
$\glue: \Pi_{(a : A)}\,\Pi_{(b:B)}\,\inl(a) = \inr(b)$.

In \cite{TheBook}, an H-space\footnote{Serre \cite{Serre1951}
  introduced H-spaces in honor of Hopf.}
in homotopy type theory is defined to consist of a
pointed type $(A,e)$ with a multiplication
$\mu : A \times A \to A$ and equalities $\lambda_a : \mu(e,a)=a$ and
$\rho_a : \mu(a,e)=a$ for all $a:A$.
However, for the Hopf construction it is useful to require that the left and right
\emph{translation maps}, $\mu(a,\blank)$ and $\mu(\blank,a)$, are
equivalences for all $a : A$. This is automatic if $A$ is connected,
but also holds, e.g., if the induced multiplication makes $\pi_0(A)$
into a group.

\begin{definition}
An \emph{H-space} is a pointed type $(A,e)$ with a multiplication 
$\mu:A\times A\to A$, and homotopies $\mu(e,\blank)\sim\mathsf{id}_A$ and
$\mu(\blank,e)\sim\mathsf{id}_A$, such that for each $a:A$ the maps
$\mu(a,\blank)$ and $\mu(\blank,a)$ are equivalences.
\end{definition}

In section 8.5.2 of \cite{TheBook}, there is a description of the Hopf construction,
which takes a connected H-space, and
produces a type family $H$ over $\susp A$ by letting the fibers over
$\north$ and $\south$ be $A$, and giving the equivalence $\mu(a,\blank)$
for the meridian $\merid(a)$. The total space is then shown to be the
join $A*A$ of $A$ with itself. The projection map $\join AA\to\susp A$ can be taken to
send the left component to $\north$, the right component to $\south$,
and for $a,b:A$ the glue path between $\inl\, a$ and $\inr\, b$ to the
meridian $\merid(\mu(a,b))$.\footnote{With the precise equivalence
  $(\Sigma_{(x : \susp A)}\,H(x)) \simeq (A * A)$ from \cite{TheBook} this
  will be mirrored.} With the described map $\join{A}{A}\to\susp{A}$, we have
a commuting triangle
\begin{equation*}
\begin{tikzcd}
\Sigma_{(x:\susp A)}\,H(x) \arrow{rr}{\simeq} \arrow{dr} & & \join{A}{A} \arrow{dl} \\
& \susp{A}.
\end{tikzcd}
\end{equation*}

Note that the only point in the Hopf construction where the connectedness of $A$ is used, is to conclude that
$\mu(a,\blank)$ and $\mu(\blank,a)$ are equivalences for each $a:A$. Hence the
Hopf construction also works if we make this requirement directly, so that it
becomes applicable in a slightly more general setting including the H-space
$\Sn^0$.

\begin{lemma}[The Hopf construction]\label{lem:hopf-construction}
Let $A$ be an H-space for which the translation maps $\mu(a,\blank)$ and
$\mu(\blank,a)$ are equivalences, for each $a:A$. Then there is a fibration
$H:\susp A\to U$ such that
\begin{equation*}
  H(\north)=H(\south)=A,\qquad\text{and}\qquad
  (\Sigma_{(x:\susp A)}\,H(x))=\join{A}{A}.
\end{equation*}
\end{lemma}
This is Lemma~8.5.7 of \cite{TheBook}, and it follows from the proof
given there, that we get a fibration sequence $A \to \join AA \to \susp
A$ where we may take the first map to be one of the inclusions.

We also recall that the join operation on types is associative (Lemma 8.5.9
of \cite{TheBook}), and that the suspension $\susp A$ of $A$ is the join $\join{\Sn^0}{A}$ (Lemma 8.5.10 of \cite{TheBook}).
In particular, it follows that $\Sn^{2n+1} \simeq \join{\Sn^n}{\Sn^n}$, for any $n:\mathbb{N}$.
To give the four Hopf fibrations in homotopy type theory, it thus suffices to give the
H-space structures on the spheres $\Sn^0$, $\Sn^1$, $\Sn^3$ and $\Sn^7$.

For $\Sn^0$, i.e., the group $\mathbb Z/2\mathbb Z$, this is trivial, and in the
case of the circle $\Sn^1$, this has already been
formalized and appears in \cite{TheBook}. In the next section we shall
see how to construct the H-space structure on $\Sn^3$.

\section{Spheroids and imaginaroids}
\label{sec:imaginaroids}

We saw in Section~\ref{sec:cayley-dickson} the classical
Cayley-Dickson construction on the level of $*$-algebras. We would obtain nothing
of interest by imitating this directly in homotopy type theory, as any real vector
space is contractible and thus equivalent to the one-point type $1$.

A first idea, which turns out to not quite work, is to give an
analog of the Cayley-Dickson construction on the level of the unit
spheres inside the $*$-algebras, as what we are ultimately after is the
H-space structure on these unit spheres. Thus we propose:
\begin{definition}
  A \define{Cayley-Dickson spheroid}\footnote{We use the term
    ``spheroid'' to emphasize that $S$ is to be thought of as a unit
    sphere, but we do not require $S$ to be an actual sphere.}
  consists of an H-space
  $S$ (we write $1$ for the base point, and concatenation denotes 
  multiplication) with additional operations
\begin{align*}
x &\mapsto x^\ast \tag{\define{conjugation}} \\
x &\mapsto -x \tag{\define{negation}}
\end{align*}
satisfying the further laws
\begin{alignat*}2
  1^* &= 1           &\qquad   (-x)^* &= -x^* \\
  -(-x)&= x = x^{**} &\qquad   x(-y) & = -xy  \\
  (xy)^* &= y^*x^*   &\qquad   x^* x & =1.
\end{alignat*}
\end{definition}

\begin{lemma}
For any two points $x$ and $y$ of a Cayley-Dickson spheroid, we have
$xx^\ast=1$ and $(-x)y=-xy$.
\end{lemma}

\begin{proof}
  For the first, simply note that $xx^\ast=x^{**}x^*=1$. For the
  second, we have:
  \begin{align*}
    (-x)y & =((-x)y)^{**} & \cdots & =(-y^*x^*)^*\\
          & =(y^*(-x)^*)^* & & =-(y^*x^*)^*\\
          & =(y^*(-x^*))^* & & =-(xy)^{**}\\
          & =\cdots & & =-xy.\qedhere
  \end{align*}
\end{proof}

The hope is now that if $S$ is an \emph{associative} Cayley-Dickson spheroid, then we can give the
join $\join SS$ the structure of a Cayley-Dickson spheroid. This turns out not quite to work, but it is instructive
to see where we get stuck.

We wish to define the multiplication $xy$ for $x,y : \join SS$ by
induction on $x$ and $y$. To do the induction on $x$ we must define
elements $(\inl\, a)y$, $(\inl\, b)y$ and dependent paths
$(\jglue\,ab)_*y : (\inl\, a)y =_{\jglue\, ab} (\inr\, b)y$ for
$a,b:S$. This is of course the same as giving the two bent arrows such
that the outer square commutes in this diagram:
\begin{equation*}
\begin{tikzcd}
S\times S\times(\join{S}{S}) \arrow{r} \arrow{d} & S\times(\join{S}{S}) \arrow{d} \arrow[bend left=20]{ddr} \\
S\times(\join{S}{S}) \arrow{r} \arrow[bend right=12]{drr} & (\join{S}{S})\times(\join{S}{S}) \arrow[densely dotted]{dr} \\
& & \join{S}{S}
\end{tikzcd}
\end{equation*}
In each case we do an induction on $y$, giving the following point
constructor problems, which we solve using equation
\eqref{eq:classical-cd}:
\begin{alignat*}2
  (\inl\, a)(\inl\, c) &\defeq \inl(ac) &\quad
  (\inl\, a)(\inr\, d) &\defeq \inr(a^*d) \\
  (\inr\, b)(\inl\, c) &\defeq \inr(cb) &\quad
  (\inr\, b)(\inr\, d) &\defeq \inl(-db^*)
\end{alignat*}
We must define four dependent paths corresponding to the interaction
of a point constructor with a path constructor, and these we all fill
with $\jglue$ (or its inverse). There results a dependent path problem
in an identity type family, which we can think of as the problem of
filling the square on the left, also depicted on the right as a
\define{diamond}:
\begin{equation}\label{eq:cd-diamond}
  \begin{tikzcd}[arrows=equals]
    \inl(ac) \arrow{r}{\jglue}\arrow{d}[swap]{\jglue} &
    \inr(cb) \arrow{d}{\rev{\jglue}} \\
    \inr(a^*d) \arrow{r}[swap]{\rev{\jglue}} &
    \inl(-db^*)
  \end{tikzcd}
  \qquad
  \begin{tikzcd}[every arrow/.append style={-},row sep=tiny,column sep=tiny]
    & cb \arrow{dl}\arrow{dr} & \\
    -db^* \arrow{dr} & & ac\arrow{dl} \\
    & a^*d &
  \end{tikzcd}
\end{equation}
These diamond shapes will play an important role in the
construction. We can define these diamond types as certain square
types sitting in a join, $\join AB$, for any $a,a':A$ and $b,b':B$:
\begin{equation}\label{eq:gen-diamond}
  \begin{tikzcd}[arrows=equals]
    a \arrow{r}{\jglue}\arrow{d}[swap]{\jglue} &
    b \arrow{d}{\rev{\jglue}} \\
    b' \arrow{r}[swap]{\rev{\jglue}} & a'
  \end{tikzcd}
  \qquad
  \begin{tikzcd}[every arrow/.append style={-},row sep=tiny,column sep=tiny]
    & b \arrow{dl}\arrow{dr} & \\
    a' \arrow{dr} & & a\arrow{dl} \\
    & b' &
  \end{tikzcd}
\end{equation}
The geometric intuition behind the shape is that we
picture the join $\join AB$ as $A$ lying on a horizontal line, $B$ on
a vertical line, and $\jglue$-paths connecting every point in $A$ to
every point in $B$.
\begin{definition}\label{defn:vhdiamond}
  Given a diamond problem corresponding to $a,a':A$ and $b,b':B$ as in
  \eqref{eq:gen-diamond}, if we have either a path $p:a=_Aa'$ or a
  path $q:b=_Bb'$, then we can solve it (i.e., fill the square on the
  left).
\end{definition}
\begin{proof}[Construction]
  By path induction on $p$ resp.\ $q$ followed by easy
  2-dimensional box filling.
\end{proof}
\begin{definition}
  Given types $A_1,A_2,B_1,B_2$ and functions $f:A_1\to A_2$ and
  $g:B_1$ to $B_2$, if we have a solution to the diamond problem in
  $\join{A_1}{B_1}$ given by $a,a':A_1$, $b,b':B_1$, then we apply
  the induced function
  $\join fg:\join{A_1}{B_1}\to \join{A_2}{B_2}$ to obtain a
  solution to the diamond problem in $\join{A_2}{B_2}$ given by
  $f\,a,f\,a':A_2$, $g\,b,g\,b':B_2$:
  \begin{equation*}
    \begin{tikzcd}[every arrow/.append style={-},row sep=tiny,column sep=tiny]
      & b \arrow{dl}\arrow{dr} & \\
      a' \arrow{dr} & & a\arrow{dl} \\
      & b' &
    \end{tikzcd}\quad\mapsto\quad
    \begin{tikzcd}[every arrow/.append style={-},row sep=tiny,column sep=tiny]
      & g\,b \arrow{dl}\arrow{dr} & \\
      f\,a' \arrow{dr} & & f\,a\arrow{dl} \\
      & g\,b' &      
    \end{tikzcd}
  \end{equation*}
\end{definition}
\begin{proof}[Construction]
  This is an instance of applying a function to a square.
\end{proof}
Coming back to \eqref{eq:cd-diamond} and fixing $a,b,c,d:S$, consider
the functions $f,g: S \to S$:
\begin{equation*}
  f(x) \defeq -acx, \qquad g(y) \defeq cyb
\end{equation*}
(we are leaving out the parentheses since we are assuming the
multiplication is associative).
\begin{lemma}\label{lem:calculations}
  If the multiplication is associative, then we have $f(-1)=ac$,
  $f(c^*a^*db^*)=-db^*$, $g(1)=cb$, and $g(c^*a^*db^*)=a^*d$.
\end{lemma}
\begin{proof}
  For example,
  \begin{alignat*}2
    ac(-c^*a^*db^*)
    &= -acc^*a^*db^* & \qquad\cdots
    &= -(aa^*)db^* \\
    &= -a(cc^*)a^*db^* &
    &= -1db^* \\
    &= -a1a^*db^* &
    &= -db^*.
  \end{alignat*}
  \par \vspace{-1.3\baselineskip}
  \qedhere
\end{proof}
Thus, it suffices to solve the diamond problem,
\begin{equation}\label{eq:x-diamond}
  \begin{tikzcd}[every arrow/.append style={-},row sep=tiny,column sep=tiny]
    & 1 \arrow{dl}\arrow{dr} & \\
    c^*a^*db^* \arrow{dr} & & -1\arrow{dl} \\
    & c^*a^*db^* &
  \end{tikzcd}\quad\text{or simply,}\quad
  \begin{tikzcd}[every arrow/.append style={-},row sep=tiny,column sep=tiny]
    & 1 \arrow{dl}\arrow{dr} & \\
    x \arrow{dr} & & -1\arrow{dl} \\
    & x &
  \end{tikzcd}
\end{equation}
with $x=c^*a^*db^*$. Naively, we might hope to solve this problem for every
$x:S$. However, considering the case where $S$ is the unit $0$-sphere
$\{\pm1\}$ in $\mathbb R$, it seems necessary to make a case
distinction on $x$ to do so. This motivates the following revised
strategy.

\subsection{Cayley-Dickson imaginaries}

Instead of just axiomatizing the unit sphere, we shall make use of the
fact that all the unit spheres in the Cayley-Dickson algebras are
suspensions of the unit sphere of imaginaries (the unit $0$-sphere in
$\mathbb R$ is of course the suspension of the $-1$-sphere, i.e., the
empty type, which corresponds to the fact the $\mathbb R$ is a
\emph{real} algebra with no imaginaries).

First we note that both conjugation and negation on Cayley-Dickson
sphere are determined by the negation acting on the imaginaries. In
fact, we can make the following general constructions:
\begin{definition}
  Suppose $A$ is a type with a negation operation. Then we can define a
  conjugation and a negation on the suspension $\susp A$ of $A$:
  \begin{alignat*}2
    \north^* &\defeq \north &   -\north &\defeq \south \\
    \south^* &\defeq \south &   -\south &\defeq \north \\
    \ap\,(\lambda x. x^*)\,(\merid\,a) &:= \merid(-a) &\quad
    \ap\,(\lambda x. {-x})\,(\merid\,a)    &:= \rev{\merid(-a)}
  \end{alignat*}
  We give $\susp A$ the base point $\north$, which we also write as
  $1$. If the negation on $A$ is involutive, then so is the
  conjugation and negation on $\susp A$.
\end{definition}

\begin{definition}
  A \define{Cayley-Dickson imaginaroid} consists of a type $A$ with an
  involutive negation, together with a binary multiplication operation
  on the suspension $\susp A$, such that $\susp A$ becomes an H-space satisfying
  the \define{imaginaroid laws}
  \begin{align*}
    x(-y) &= -xy \\
    xx^* & =1 \\
    (xy)^* &= y^*x^*
  \end{align*}
  for $x,y:\susp A$.
\end{definition}
Note that if $A$ is a Cayley-Dickson imaginaroid, then $\susp A$
becomes a Cayley-Dickson spheroid.

\begin{definition}\label{def:imag-h-space}
Let $A$ be a Cayley-Dickson imaginaroid where the multiplication on
$\susp A$ is associative. Then $A' \defeq \join{\susp A}{\susp A}$ can
be given the structure of an H-space.
\end{definition}
\begin{proof}[Construction]
We can define the multiplication on $\join{\susp A}{\susp A}$
as in the previous section, leading to the
diamond problem~\eqref{eq:x-diamond}. This we now solve by induction
on $x:\susp A$. The diamonds for the poles are easily filled using
Definition~\ref{defn:vhdiamond}:
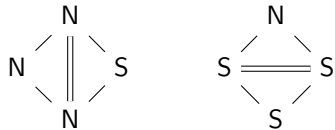
\begin{equation*}
  \begin{tikzcd}[every arrow/.append style={-},row sep=tiny,column sep=tiny]
    & \north \arrow{dl}\arrow{dr}\arrow[-, double equal sign distance]{dd} & \\
    \north \arrow{dr} & & \south\arrow{dl} \\
    & \north &
  \end{tikzcd}\qquad
  \begin{tikzcd}[every arrow/.append style={-},row sep=tiny,column sep=tiny]
    & \north \arrow{dl}\arrow{dr} & \\
    \south \arrow{dr}\arrow[-, double equal sign distance]{rr} & & \south\arrow{dl} \\
    & \south &
  \end{tikzcd}
\end{equation*}
These solutions must now be connected by filling, for every $a : A$,
the following hollow cube connecting the diamonds:
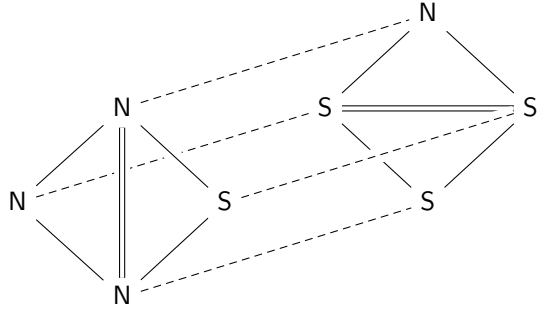
\begin{equation*}
  \begin{tikzcd}[every arrow/.append style={-}]
    &   &   &   & \north\arrow{dl}\arrow{dr} & \\
    & \north\arrow{dl}\arrow{dr}\arrow[dashed]{urrr} &
    & \south\arrow{dr}\arrow[-, double equal sign distance]{rr} & & \south\arrow{dl} \\
    \north\arrow{dr}\arrow[dashed]{urrr} & &
    \south\arrow[crossing over]{ul}\arrow{dl}\arrow[dashed,crossing over]{urrr} &
    & \south & \\
    & \north\arrow[-, double equal sign distance,crossing over]{uu}\arrow[dashed]{urrr} & & & &
  \end{tikzcd}
\end{equation*}
Here, the two dashed paths $\north=\north$ and $\south=\south$ are
identities, while the other two are each the meridian,
$\merid\,a:\north=\south$. Generalizing a bit, we see that we can fill
any cube in a symmetric join, $\join BB$, with $p:x=_By$, of this form:
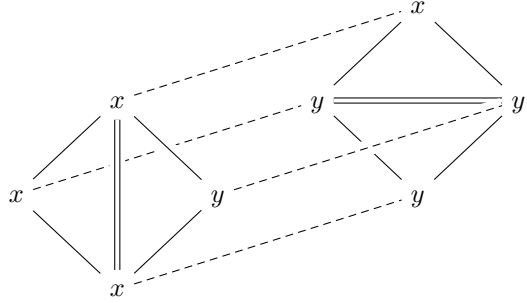
\begin{equation*}
  \begin{tikzcd}[every arrow/.append style={-}]
    &   &   &   & x\arrow{dl}\arrow{dr} & \\
    & x\arrow{dl}\arrow{dr}\arrow[dashed]{urrr} &
    & y\arrow{dr}\arrow[-, double equal sign distance]{rr} & & y\arrow{dl} \\
    x\arrow{dr}\arrow[dashed]{urrr} & &
    y\arrow[crossing over]{ul}\arrow{dl}\arrow[dashed,crossing over]{urrr} &
    & y & \\
    & x\arrow[-, double equal sign distance,crossing over]{uu}\arrow[dashed]{urrr} & & & &
  \end{tikzcd}
\end{equation*}
Indeed, this follows by path induction on $p$ followed by trivial
manipulations.

This multiplication has the virtue that the H-space laws $1x=x1=x$ are
very easy to prove; indeed, for point constructors they follow from
the H-space laws on $\susp A$, and since these point constructors land
in the two different sides of the join, we can glue them together
trivially on path constructors.
\end{proof}

\begin{conjecture}
  Suppose $A$ is a Cayley-Dickson imaginaroid where the multiplication on
  $\susp A$ is associative and some further {\normalfont(}for now
  unspecified\/{\normalfont)}
  coherence conditions obtain. Then $A'\defeq\join{A}{\susp A}$ can
  also be given the structure of a Cayley-Dickson imaginaroid, which
  is associative if $A$ is furthermore commutative.
\end{conjecture}
We get of course a negation on $A'$ in a canonical way using the
negations on $A$ and $\susp A$. Using associativity of join and the
fact that $\join{\Sn^0}X=\susp X$ for any $X$, we get
$\susp A'=\join{\Sn^0}{(\join A{\susp A})} =
\join{(\join{\Sn^0}A)}{\susp A}=\join{\susp A}{\susp A}$. Thus the
multiplication on $\susp A$ comes from the previous
construction. The hard part is then to verify the algebraic laws,
which is where we expect that coherence conditions on the algebraic
structure for $A$ will come in.

Let us finish this section by stating the result of combining
the Hopf construction (Lemma~\ref{lem:hopf-construction}) and the
H-space structure on $\Sn^3$, which we obtain from
Definition~\ref{def:imag-h-space} using the obvious imaginaroid
structure on $\Sn^0$ and the associativity of the H-space structure on
$\Sn^1=\susp\Sn^0$:
\begin{theorem}
  There is a fibration sequence
  \[ \Sn^3 \to \Sn^7 \to \Sn^4 \]
  of pointed maps.
\end{theorem}
\begin{corollary}
  There is an element of infinite order in $\pi_7(\Sn^4)$.
\end{corollary}
\begin{proof}
  Consider the long exact sequence of homotopy groups
  \cite[Theorem~8.4.6]{TheBook} corresponding to the above fibration
  sequence. In particular, we get the exactness of
  \[
    \pi_7(\Sn^3) \to \pi_7(\Sn^7) \to \pi_7(\Sn^4).
  \]
  The inclusion of the fiber, $\Sn^3 \hookrightarrow
  \join{\Sn^3}{\Sn^3} = \Sn^7$, is nullhomotopic, so the first map is
  zero. Since $\pi_7(\Sn^7) = \Z$, we get an exact sequence
  \[
    0 \to \Z \to \pi_7(\Sn^4),
  \]
  which gives the desired element of infinite order.
\end{proof}

\section{Semantics}
\label{sec:semantics}

One expects that anything that is done in homotopy type theory, can also be
done in most $(\infty,1)$-toposes. However, general $(\infty,1)$-topos semantics of 
homotopy type theory is currently still conjectural. 
Nonetheless, there is semantics for homotopy type theory in the usual 
$(\infty,1)$-topos of $\infty$-groupoids (in terms of simplicial sets 
\cite{KapulkinLumsdaine2012}, and in cubical sets \cite{BCH2014,CCHM2016}),
and in certain presheaf $(\infty,1)$-toposes \cite{Shulman2013,ShulmanEI}.

On the other hand, given a particular construction in homotopy type theory,
one can investigate what semantics is needed to perform the construction in other
`homotopy theories', for instance in $(\infty,1)$-categories presented by (Quillen) model categories. 
An example of this kind is given by \cite{Rezk}, who translated the formalized
proof of the Blakers-Massey theorem to obtain a new, purely homotopy theoretic
proof in the category of spaces.
In this section, we describe what seems to be needed to perform (i) the construction
of the H-space structure on $\Sn^3$ (Section \ref{sec:imaginaroids}),
and (ii) the Hopf construction (Lemma~\ref{lem:hopf-construction}).

The Hopf construction requires some form of univalence, for instance
an object classifier as in an $(\infty,1)$-topos.
For any H-space $A$ we always have a
map $\join AA \to \susp A$, but in general the homotopy fiber may fail to
be $A$ (consider, e.g., $\Sn^0$ in the category of sets equipped with the
trivial model structure).

Observe that for the construction of the H-space structure on $\Sn^3$, we have
used only a small fragment of homotopy type theory. 
We have used dependent sums, identity types, and homotopy pushouts. A priori
we also use the inductive families of squares and cubes (of paths in a type),
but these can be equivalently defined in terms of identity types, see the 
next subsection \ref{sec:machinery}.

In general, to model dependent sums and identity types in a Quillen model category,
some extra coherence is needed \cite{AwodeyWarren2009,vdBergGarner2012}. 
However, to reproduce a particular type theoretic construction, this extra coherence may not be needed.
Since a Quillen model category has homotopy pushouts, an empty space and a unit space, it also has the $n$-spheres.
The construction corresponding to Definition~\ref{def:imag-h-space}
only uses finite homotopy colimits and their universal properties. 
Therefore, we expect that the construction of the H-space structure on $\Sn^3$
is possible in any Quillen model category.

\subsection{The cubical machinery}\label{sec:machinery}

In the formalization we use the cubical methods of
\cite{LicataBrunerie2015}, which consists in using inductively defined
families of square, cubes, squareovers, etc.
These are available in any model category (up to pullback stability), because there are
alternative definitions just in terms of identity types and dependent
sums.

Consider for instance the type of squares in a type $A$. These are
parameterized by the top-left corner $a_{00} : A$. The dependent sum
type
\begin{multline*}
  B := 
  \Sigma_{(a_{02} : A)}\,\Sigma_{(a_{20} : A)}\,\Sigma_{(a_{22} : A)}\, \\
  (a_{00} = a_{02})\times(a_{20} = a_{22})\times (a_{00} = a_{20}) \times (a_{02} = a_{22})
\end{multline*}
describes the type of boundaries of squares in $A$ with top-left
corner $a_{00}$. There is an obvious element $r := (a_{00}, a_{00}, a_{00},
1_{a_{00}}, 1_{a_{00}}, 1_{a_{00}}, 1_{a_{00}})$ representing the
trivial boundary. Now the type of squares with boundary $b : B$ can be
represented simply as the identity type $(b = r)$. The representation
of cubes and squareovers proceeds in a similar manner.

We are grateful to Christian Sattler for this observation, which
derives from considerations of the Reedy fibrant replacement of the
constant diagram over the semi-cubical indexing category.

\section{Conclusion}
\label{sec:conclusion}

One might also wonder whether our construction applies to other H-spaces
in the usual homotopy category besides the spheres $\Sn^0$, $\Sn^1$, and
$\Sn^3$, in other words, what are the associative imaginaroids in
ordinary homotopy theory?

We are grateful to Mark Grant and Qiaochu Yuan for the following
observations (in response to a question on MathOverflow
\cite{MOHspaces}). If a space $X$ is a suspension, then it is
automatically a co-H-space, and \cite{West1972} proved that the only
finite complexes which are both H-spaces and co-H-spaces are the
spheres $\Sn^0$, $\Sn^1$, $\Sn^3$ and $\Sn^7$. Beyond the finite complexes,
note that the rationalization $\Sn^{2n+1}_{\mathbb Q}$ of any odd-dimensional sphere is an
associative H-space that is also a suspension, but in this case we
already know that the join $\join{\Sn^{2n+1}_{\mathbb
    Q}}{\Sn^{2n+1}_{\mathbb Q}} \simeq \Sn^{4n+3}_{\mathbb Q}$ is again an
H-space.
It remains to be seen whether there are non-trivial applications in
other homotopy theories.

The classical Cayley-Dickson construction gives more than just the
H-space structure on $\Sn^3$, namely it presents $\Sn^3$ as the
topological group $Sp(1)$ (which is also $SU(2)$). Topological groups
can be represented in homotopy type theory via their classifying
types, but we do not know how to define a delooping of $\Sn^3$ in
homotopy type theory (classically this would be the
infinite-dimensional quaternionic projective space
$\mathbb{H}\mathrm{P}^\infty$).

One of the other fascinating aspects of the classical Cayley-Dickson
construction is of course that it can be iterated. Our construction
as it stands does not allow for iteration, and of course we can not
expect it to be indefinitely applicable as we need the associativity
condition. However, it is conceivable that for a strengthened notion of
imaginaroid $A$ (including some coherence conditions on the algebraic
structure), we could have that $\join A{\susp A}$ is again an
imaginaroid. This would be one way to obtain the H-space structure on
$\Sn^7$ in homotopy type theory, but we leave this to future work.

Another byproduct of the classical Cayley-Dickson construction is that
we find the exceptional Lie group $G_2$ as the zero divisors in the
sedenions. Unfortunately, there seems to be no hope for our current approach to
yield such fruits.

\bibliographystyle{hplain}
\bibliography{ktheory-hott}

\section*{Appendix: formalization}

The following files have been incorporated into the homotopy type
theory library of Lean, and can thus be found at:
\url{https://github.com/leanprover/lean/}.

File \verb"imaginaroid.hlean":
\begin{lstlisting}
import algebra.group cubical.square types.pi .hopf

open eq eq.ops equiv susp hopf
open [notation] sum

namespace imaginaroid

structure has_star [class] (A : Type) :=
(star : A → A)

reserve postfix `*` : (max+1)
postfix `*` := has_star.star

structure involutive_neg [class] (A : Type) extends has_neg A
  := (neg_neg : ∀a, neg (neg a) = a)

section
  variable {A : Type}
  variable [H : involutive_neg A]
  include H

  theorem neg_neg (a : A) : - -a = a :=
  !involutive_neg.neg_neg
end

section
  /- In this section we construct, when A has a negation,
     a unit, a negation and a conjugation on susp A.
     The unit 1 is north, so south is -1. The negation must
     then swap north and south, while the conjugation fixes
     the poles and negates on meridians.
  -/
  variable {A : Type}

  definition has_one_susp [instance] : has_one (susp A) :=
  ⦃ has_one, one := north ⦄

  variable [H : has_neg A]
  include H

  definition susp_neg : susp A → susp A :=
  susp.elim south north (λa, (merid (neg a))⁻¹)

  definition has_neg_susp [instance] : has_neg (susp A) :=
  ⦃ has_neg, neg := susp_neg⦄

  definition susp_star : susp A → susp A :=
  susp.elim north south (λa, merid (neg a))

  definition has_star_susp [instance] : has_star (susp A) :=
  ⦃ has_star, star := susp_star ⦄
end

section
  -- If negation on A is involutive, so is negation on susp A
  variable {A : Type}
  variable [H : involutive_neg A]
  include H

  definition susp_neg_neg (x : susp A) : - - x = x :=
  begin
    induction x with a,
    { reflexivity },
    { reflexivity },
    { apply eq_pathover, rewrite ap_id,
      rewrite (ap_compose' (λy, -y)),
      krewrite susp.elim_merid, rewrite ap_inv,
      krewrite susp.elim_merid, rewrite neg_neg,
      rewrite inv_inv, apply hrefl }
  end

  definition involutive_neg_susp [instance]
    : involutive_neg (susp A) :=
  ⦃ involutive_neg, neg_neg := susp_neg_neg ⦄

  definition susp_star_star (x : susp A) : x** = x :=
  begin
    induction x with a,
    { reflexivity },
    { reflexivity },
    { apply eq_pathover, rewrite ap_id,
      krewrite (ap_compose' (λy, y*)),
      do 2 krewrite susp.elim_merid, rewrite neg_neg,
      apply hrefl }
  end

  definition susp_neg_star (x : susp A) : (-x)* = -x* :=
  begin
    induction x with a,
    { reflexivity },
    { reflexivity },
    { apply eq_pathover,
      krewrite [ap_compose' (λy, y*),
                ap_compose' (λy, -y) (λy, y*)],
      do 3 krewrite susp.elim_merid, rewrite ap_inv,
      krewrite susp.elim_merid, apply hrefl }
  end
end

structure imaginaroid [class] (A : Type)
  extends involutive_neg A, has_mul (susp A) :=
(one_mul : ∀x, mul one x = x)
(mul_one : ∀x, mul x one = x)
(mul_neg :
  ∀x y, mul x (@susp_neg A ⦃ has_neg, neg := neg ⦄ y)
        = @susp_neg A ⦃ has_neg, neg := neg ⦄ (mul x y))
(norm :
  ∀x, mul x (@susp_star A ⦃ has_neg, neg := neg ⦄ x) = one)
(star_mul :
  ∀x y, @susp_star A ⦃ has_neg, neg := neg ⦄ (mul x y)
        = mul (@susp_star A ⦃ has_neg, neg := neg ⦄ y)
              (@susp_star A ⦃ has_neg, neg := neg ⦄ x))

section
  variable {A : Type}
  variable [H : imaginaroid A]
  include H

  theorem one_mul (x : susp A) : 1 * x = x :=
  !imaginaroid.one_mul
  
  theorem mul_one (x : susp A) : x * 1 = x :=
  !imaginaroid.mul_one
  
  theorem mul_neg (x y : susp A) : x * -y = -x * y :=
  !imaginaroid.mul_neg

  /- this should not be an instance because we typically
     construct the h_space structure on susp A before
     defining the imaginaroid structure on A -/
  definition imaginaroid_h_space : h_space (susp A) :=
  ⦃ h_space, one := one, mul := mul,
    one_mul := one_mul, mul_one := mul_one ⦄

  theorem norm (x : susp A) : x * x* = 1 :=
  !imaginaroid.norm
  
  theorem star_mul (x y : susp A) : (x * y)* = y* * x* :=
  !imaginaroid.star_mul
  
  theorem one_star : 1* = 1 :> susp A := idp

  theorem neg_mul (x y : susp A) : (-x) * y = -x * y :=
  calc
    (-x) * y = ((-x) * y)**  : susp_star_star
         ... = (y* * (-x)*)* : star_mul
         ... = (y* * -x*)*   : susp_neg_star
         ... = (-y* * x*)*   : mul_neg
         ... = -(y* * x*)*   : susp_neg_star
         ... = -x** * y**    : star_mul
         ... = -x * y**      : susp_star_star
         ... = -x * y        : susp_star_star

  theorem norm' (x : susp A) : x* * x = 1 :=
  calc
    x* * x = (x* * x)**  : susp_star_star
       ... = (x* * x**)* : star_mul
       ... = 1*          : norm
       ... = 1           : one_star
end

/- Here we prove that if A has an associative imaginaroid
   structure, then join (susp A) (susp A) is an h_space -/
section
  parameter A : Type
  parameter [H : imaginaroid A]
  parameter (assoc : Πa b c : susp A, (a * b) * c = a * b * c)
  include A H assoc

  open join

  section lemmata
    parameters (a b c d : susp A)

    local abbreviation f : susp A → susp A :=
    λx, a * c * (-x)

    local abbreviation g : susp A → susp A :=
    λy, c * y * b

    definition lemma_1 : f (-1) = a * c :=
    calc
      a * c * (- -1) = a * c * 1  : idp
                 ... = a * c      : mul_one

    definition lemma_2 : f (c* * a* * d * b*) = - d * b* :=
    calc
      a * c * (-c* * a* * d * b*)
        = a * (-c * c* * a* * d * b*)     : mul_neg
    ... = -a * c * c* * a* * d * b*       : mul_neg
    ... = -(a * c) * c* * a* * d * b*     : assoc
    ... = -((a * c) * c*) * a* * d * b*   : assoc
    ... = -(a * c * c*) * a* * d * b*     : assoc
    ... = -(a * 1) * a* * d * b*          : norm
    ... = -a * a* * d * b*                : mul_one
    ... = -(a * a*) * d * b*              : assoc
    ... = -1 * d * b*                     : norm
    ... = -d * b*                         : one_mul

    definition lemma_3 : g 1 = c * b :=
    calc
      c * 1 * b = c * b : one_mul

    definition lemma_4 : g (c* * a* * d * b*) = a* * d :=
    calc
      c * (c* * a* * d * b*) * b
        = (c * c* * a* * d * b*) * b      : assoc
    ... = ((c * c*) * a* * d * b*) * b    : assoc
    ... = (1 * a* * d * b*) * b           : norm
    ... = (a* * d * b*) * b               : one_mul
    ... = a* * (d * b*) * b               : assoc
    ... = a* * d * b* * b                 : assoc
    ... = a* * d * 1                      : norm'
    ... = a* * d                          : mul_one
  end lemmata

  /- in the algebraic form, the Cayley-Dickson
     multiplication has:

      (a,b) * (c,d) = (a * c - d * b*, a* * d + c * b)

    Here we do the spherical/imaginaroid form.
  -/
  definition cd_mul (x y : join (susp A) (susp A))
    : join (susp A) (susp A) :=
  begin
    induction x with a b a b,
    { induction y with c d c d,
      { exact inl (a * c) },
      { exact inr (a* * d) },
      { apply glue }
    },
    { induction y with c d c d,
      { exact inr (c * b) },
      { exact inl (- d * b*) },
      { apply inverse, apply glue }
    },
    { induction y with c d c d,
      { apply glue },
      { apply inverse, apply glue },
      { apply eq_pathover,
        krewrite [join.elim_glue,join.elim_glue],
        change join.diamond (a * c) (-d * b*) (c * b) (a* * d),
        rewrite [-(lemma_1 a c),-(lemma_2 a b c d),
                 -(lemma_3 b c),-(lemma_4 a b c d)],
        apply join.ap_diamond (f a c) (g b c),
        generalize (c* * a* * d * b*), clear a b c d,
        intro x, induction x with i,
        { apply join.vdiamond, reflexivity },
        { apply join.hdiamond, reflexivity },
        { apply join.twist_diamond } } }
  end

  definition cd_one_mul (x : join (susp A) (susp A))
    : cd_mul (inl 1) x = x :=
  begin
    induction x with a b a b,
    { apply ap inl, apply one_mul },
    { apply ap inr, apply one_mul },
    { apply eq_pathover, rewrite ap_id, unfold cd_mul,
      krewrite join.elim_glue, apply join.hsquare }
  end

  definition cd_mul_one (x : join (susp A) (susp A))
    : cd_mul x (inl 1) = x :=
  begin
    induction x with a b a b,
    { apply ap inl, apply mul_one },
    { apply ap inr, apply one_mul },
    { apply eq_pathover, rewrite ap_id, unfold cd_mul,
      krewrite join.elim_glue, apply join.hsquare }
  end

  definition cd_h_space [instance]
    : h_space (join (susp A) (susp A)) :=
  ⦃ h_space, one := inl one, mul := cd_mul,
    one_mul := cd_one_mul, mul_one := cd_mul_one ⦄

end

end imaginaroid
\end{lstlisting}

File \verb"quaternionic_hopf.hlean"

\begin{lstlisting}
import .complex_hopf .imaginaroid

open eq equiv is_equiv circle is_conn trunc is_trunc
     sphere_index sphere susp imaginaroid

namespace hopf

definition involutive_neg_empty [instance]
  : involutive_neg empty :=
⦃ involutive_neg, neg := empty.elim,
  neg_neg := by intro a; induction a ⦄

definition involutive_neg_circle [instance]
  : involutive_neg circle :=
involutive_neg_susp

definition has_star_circle [instance] : has_star circle :=
has_star_susp

/- this is the "natural" conjugation defined using the
   base-loop recursor -/
definition circle_star [reducible] : S¹ → S¹ :=
circle.elim base loop⁻¹

definition circle_neg_id (x : S¹) : -x = x :=
begin
  fapply (rec2_on x),
  { exact seg2⁻¹ },
  { exact seg1 },
  { apply eq_pathover, rewrite ap_id, krewrite elim_merid,
    apply square_of_eq, reflexivity },
  { apply eq_pathover, rewrite ap_id, krewrite elim_merid,
    apply square_of_eq, apply trans (con.left_inv seg2),
    apply inverse, exact con.left_inv seg1 }
end

definition circle_mul_neg (x y : S¹) : x * (-y) = - x * y :=
by rewrite [circle_neg_id,circle_neg_id]

definition circle_star_eq (x : S¹) : x* = circle_star x :=
begin
  fapply (rec2_on x),
  { reflexivity },
  { exact seg2⁻¹ ⬝ (tr_constant seg1 base)⁻¹ },
  { apply eq_pathover, krewrite elim_merid,
    rewrite elim_seg1, apply square_of_eq, apply trans
      (ap (λw, w ⬝ (tr_constant seg1 base)⁻¹)
          (con.right_inv seg2)⁻¹),
    apply con.assoc },
  { apply eq_pathover, krewrite elim_merid,
    rewrite elim_seg2, apply square_of_eq,
    rewrite [↑loop,con_inv,inv_inv,idp_con],
    apply con.assoc }
end

open prod prod.ops

definition circle_norm (x : S¹) : x * x* = 1 :=
begin
  rewrite circle_star_eq, induction x,
  { reflexivity },
  { apply eq_pathover, rewrite ap_constant,
    krewrite [ap_compose' (λz : S¹ × S¹, circle_mul z.1 z.2)
                          (λa : S¹, (a, circle_star a))],
    rewrite [ap_compose' (prod_functor (λa : S¹, a)
                                       circle_star)
                         (λa : S¹, (a, a))],
    rewrite ap_diagonal,
    krewrite [ap_prod_functor (λa : S¹, a) circle_star
                              loop loop],
    rewrite [ap_id,↑circle_star], krewrite elim_loop,
    krewrite (ap_binary circle_mul loop loop⁻¹),
    rewrite [ap_inv,↑circle_mul,elim_loop,ap_id,
             ↑circle_turn,con.left_inv],
    constructor }
end

definition circle_star_mul (x y : S¹) : (x * y)* = y* * x* :=
begin
  induction x,
  { apply inverse, exact circle_mul_base (y*) },
  { apply eq_pathover, induction y,
    { exact natural_square_tr 
        (λa : S¹, ap (λb : S¹, b*) (circle_mul_base a))
        loop },
    { apply is_prop.elimo } }
end

definition imaginaroid_sphere_zero [instance]
  : imaginaroid (sphere (-1.+1)) :=
⦃ imaginaroid,
  neg_neg := susp_neg_neg,
  mul := circle_mul,
  one_mul := circle_base_mul,
  mul_one := circle_mul_base,
  mul_neg := circle_mul_neg,
  norm := circle_norm,
  star_mul := circle_star_mul ⦄

local attribute sphere [reducible]
open sphere.ops

definition sphere_three_h_space [instance] : h_space (S 3) :=
@h_space_equiv_closed (join S¹ S¹)
  (cd_h_space (S -1.+1) circle_assoc)
  (S 3) (join.spheres 1 1)

definition is_conn_sphere_three : is_conn 0 (S 3) :=
begin
  have le02 : trunc_index.le 0 2,
  from trunc_index.le.step
    (trunc_index.le.step (trunc_index.le.tr_refl 0)),
  exact @is_conn_of_le (S 3) 0 2 le02 (is_conn_sphere 3)
end

local attribute is_conn_sphere_three [instance]

definition quaternionic_hopf : S 7 → S 4 :=
begin
  intro x, apply @sigma.pr1 (susp (S 3)) (hopf (S 3)),
  apply inv (hopf.total (S 3)), apply inv (join.spheres 3 3),
  exact x
end

open pointed fiber function

definition quaternionic_phopf [constructor] : S* 7 →* S* 4 :=
proof pmap.mk quaternionic_hopf idp qed

definition pfiber_quaternionic_phopf
  : pfiber quaternionic_phopf ≃* S* 3 :=
begin
  fapply pequiv_of_equiv,
  { esimp, unfold [quaternionic_hopf],
    refine fiber.equiv_precompose
      (sigma.pr1 ∘ (hopf.total (S 3))⁻¹ᵉ)
      (join.spheres (of_nat 3) (of_nat 3))⁻¹ᵉ _ ⬝e _,
    refine fiber.equiv_precompose _
      (hopf.total (S 3))⁻¹ᵉ _ ⬝e _,
    apply fiber_pr1 },
  { reflexivity }
end

end hopf
\end{lstlisting}
\end{document}